\documentclass[a4paper,11pt]{article}

\usepackage{amsmath}
\usepackage{amssymb}
\usepackage{amsthm}

\newtheorem{thm}{Theorem}[section]
\newtheorem{lemma}[thm]{Lemma}
\newtheorem{prop}[thm]{Proposition}
\newtheorem{cor}[thm]{Corollary}

\newtheorem{defn}[thm]{Definition}

\newtheorem{rmk}[thm]{Remark}

\newcommand{\mR}{\mathbb{R}}
\newcommand{\XX}{\mathcal{X}}

\newcommand{\dist}{\operatorname{dist}}

\begin{document}
\title{Almost sure convergence of maxima for chaotic dynamical systems.}
\author{M.P. Holland, M. Nicol, A. T\"or\"ok}
\date{\today}

\maketitle

\begin{abstract}
Suppose $(f,\XX,\nu)$ is a measure preserving dynamical system and  $\phi:\XX\to\mathbb{R}$ is an observable with some degree of regularity.
We investigate  the maximum process $M_n:=\max\{X_1,\ldots,X_n\}$, where $X_i=\phi\circ f^i$ is a time series of observations  on  the system.
When $M_n\to\infty$ almost surely, we establish results on the almost sure growth rate, namely the existence (or otherwise) of a sequence $u_n\to\infty$ such that $M_n/u_n\to 1$
almost surely.  The observables we consider will be functions of the distance to a distinguished point $\tilde{x}\in \XX$. Our results are based on the interplay between shrinking target problem estimates at $\tilde{x}$ and the form of the observable (in particular polynomial or logarithmic) near $\tilde{x}$. We determine where such an almost sure limit exists and give examples where it does not.  
Our results apply to a wide class of non-uniformly hyperbolic dynamical systems, 
under mild assumptions on the rate of mixing, and on regularity of the invariant measure. 
%Applications of our
%results include non-uniformly expanding systems (such intermittency-type maps), and non-uniformly hyperbolic systems which include dispersing billiard maps,  Lozi type maps, H\'enon-like %diffeomorphisms and Lorenz-like  maps. 
 \end{abstract}

\section{Introduction}\label{sec.introduction} 
Let $(f,\XX, \nu)$ be a dynamical system, where $\XX\subset\mathbb{R}$, $f:\XX\rightarrow \XX$ is a measurable transformation, and $\nu$ is an $f$-invariant probability measure supported on $\XX$. Given an observable
$\phi:\XX \rightarrow \mR$, i.e. a measurable function,  we consider the stationary stochastic process
$X_1, X_2, \dots$ defined as
\begin{equation}\label{eq:phi_o_f} 
X_i =\phi \circ f^{i-1}, \quad i \geq 1,  
\end{equation}
and its associated maximum process $M_n$ defined as
\begin{equation}
\label{eq:max-process} 
M_n = \max(X_1,\dots,X_n). 
\end{equation}  
 
Recent research has investigated the distributional behavior of $M_n$
and in particular the existence of
sequences $a_n,b_n\in\mathbb{R}$ such that
\begin{equation}\label{eq:extreme-conv1}
\nu\left\{x\in\XX: a_n(M_n-b_n) \leq u \right\}\to G(u),
\end{equation}
for some non-degenerate distribution function  $G(u)$, $-\infty <u <\infty$. These results have shown that for sufficiently hyperbolic systems and for regular enough observables $\phi$ maximized at generic  points $\tilde{x}$, the distribution limit is the 
same as that  which would 
hold if $\{X_i\}$ were  independent identically distributed (i.i.d.) random
variables with the same distribution function as $\phi$~\cite{FFT1,GHN,HNT}. As in the classical situation $G(u)$ can be one of the three extreme value distributions Type I (Gumbell), Type II (Frech\'et) and Type III (Weibull). If $\tilde{x}$ is periodic we expect different behavior (for details see~\cite{CFFHN,FFT2,Ferguson_Pollicott, Keller}). 

In~\cite[Lemma 1.1]{GHN} the elementary observation is made:

\begin{prop}\label{thm.log-vs-powers}
  Assume  a function $x\mapsto g(x)$  has a minimum value of zero at a unique point $\tilde{x} \in \XX$.
 % (we have in mind the function $g(x) = d(x,\tilde)$).
  % , or a function having finitely many minima of the same value (if
  % finite)

  The following are equivalent, where $\alpha > 0$:
  \begin{enumerate}
  \item A Type I law for $x\mapsto -\log d(x,\tilde{x})$ with $a_n=1$ and
    $b_n=\log n$;
  \item A Type II law for $x\mapsto d(x,\tilde{x})^{-\alpha}$ with $a_n=
    n^{-\alpha}$ and $b_n=0$;
  \item A Type III law for $x\mapsto {C}- d(x,\tilde{x})^{\alpha}$ with
    $a_n= n^{\alpha}$ and $b_n={C}$;
  \end{enumerate}
  (and similarly for other choices of $b_n$ in the first case).

\end{prop}

This paper will consider the almost sure behavior of $M_n$.  A fundamental problem is to determine the existence (or otherwise) of a sequence $u_n\to\infty$ such that $M_n/u_n\to 1$
(almost surely) so as to determine almost sure rates of growth of $M_n$ (perhaps with error bounds). Motivated by Proposition~\ref{thm.log-vs-powers}  and other applications we will consider the functions
$-\log d(x,\tilde{x})$ and $d(x,\tilde{x})^{-\alpha}$. 
If $\nu$ is ergodic and $\phi$ is essentially bounded then almost surely, $M_n\to \mbox {ess} \sup \phi$, hence the limit of $M_n$ in the case of a bounded observable (Type III) is clear. 
 
For independent, identically distributed (i.i.d) random variables the almost sure behavior of $M_n$  has been widely studied, e.g. in the subject area of \emph{extreme value theory} \cite{Embrechts,Galambos}. However within a dynamical systems framework, and also for general dependent random variables less is known about almost sure growth rates of $M_n$. In this article we determine the existence (or otherwise) of sequences $u_n\to\infty$ such that $M_n/u_n\to 1$ for the functions $-\log d(x,\tilde{x})$ and $d(x,\tilde{x})^{-\alpha}$. 

In the  case where $\phi(x)=-\log\dist(x,\tilde{x})$ for given $\tilde{x}\in\XX$, we show for a broad class of chaotic systems  that for
$\nu$-a.e. $\tilde{x}\in\XX$ 
$$\lim_{n\to\infty}\frac{M_n(x)}{\log n}=\frac{1}{d_{\nu}},\quad\textrm{a.s.},$$
where $d_{\nu}$ is the local dimension of the measure $\nu$.
Towards proving almost sure convergence of $M_n/u_n$ we will establish Borel-Cantelli results for non-uniformly hyperbolic systems and extend results of~\cite{HNPV}. We will also consider observables of the form $\phi(x)=\dist(x,\tilde{x})^{-\alpha}$, $\alpha>0$. In this latter case the almost sure limit $M_n/u_n$ does not necessarily exist for any sequence $u_n$, and instead we  give growth rate bounds.

We organise this paper as follows. In Section \ref{sec.results} we  state the main assumptions placed on the dynamical systems. Such assumptions will be phrased in terms of i) the form of the observable $\phi:\XX\to\mathbb{R}$ (which will always be assumed to have a unique maxima at a point $\tilde{x}$); ii) the regularity of the invariant measure $\nu$ and the  iii) the rate of mixing. Our proofs use recent work on dynamical Borel Cantelli lemmas. In Section \ref{sec.general} we establish almost sure convergence of $M_n$ for a wide class of chaotic systems under these
assumptions. In particular we  demonstrate the sensitivity of the convergence to the form of the observable and the regularity of the measure. In later sections we make assumptions on return time statistics
and obtain more refined estimates on the almost sure behavior of $M_n$.
statistics. Finally we compare our findings to what is known in the i.i.d case. This latter work builds upon that of \cite{GNO}.

\section{Statement of results}\label{sec.results}
In this section we consider a measure preserving system $(f,\XX,\nu)$, where $\XX\subset\mathbb{R}^d$, and $\nu$ is an ergodic Sinai-Ruelle-Bowen (SRB) measure. 
%The systems we consider include those that can be modelled
%by a Young tower \cite{Young}, but in the statement of our results we just require control on the rate of decay of correlations, and control on the regularity of $\nu$. 
We first make  precise our notion of decay of correlations.
\begin{defn}
We say that $(f,\XX,\nu)$ has decay of correlations in $\mathcal{B}_1$ versus $\mathcal{B}_2$ (where $\mathcal{B}_1$ and $\mathcal{B}_2$ are Banach spaces)
with rate function $\Theta(j)\to 0$ if for all $\varphi_1\in\mathcal{B}_1$ and $\varphi_2\in\mathcal{B}_2$ we have:
\[
\mathcal{C}_j(\varphi_1,\varphi_2,\nu):=\left|\int \varphi_1 \cdot \varphi_2\circ f^j d \nu -\int \varphi_1 d \nu \int \varphi_2 d\nu\right|
\leq
\Theta(j) \|\varphi_1\|_{\mathcal{B}_1} \|\varphi_2\|_{\mathcal{B}_2},
\]
where $\|\cdot\|_{\mathcal{B}_i}$ denotes the corresponding norms on the Banach spaces.
\end{defn}
Given this definition we state the following assumption: 
\begin{enumerate}
\item[(A1)]{\bf (Decay of correlations).}
There exists a monotonically decreasing sequence $\Theta(j)\to 0$ such that
for all Lipschitz $\varphi_1$ and $\varphi_2$:
\[
\mathcal{C}_j(\varphi_1,\varphi_2,\nu)\leq
\Theta(j) \|\varphi_1\|_{\textrm{Lip}} \|\varphi_2\|_{\textrm{Lip}},
\]
where $\|\cdot\|_{\textrm{Lip}}$ denotes the Lipschitz norm.
\end{enumerate}
For a wide class of non-uniformly hyperbolic systems (including those with stable foliations) it is known that condition
(A1) holds with estimates on $\Theta(j)$. We also consider variants of (A1), such as decay of correlations for \emph{Lipschitz versus $L^{\infty}$}
or decay for \emph{BV versus $L^1$}, where BV is the space of functions of bounded variation. For these latter conditions, we can make stronger statements about almost sure convergence rates for $M_n$, but these conditions are in general  valid only for a more restricted class of dynamical system. 

We now consider the regularity of the measure $\nu$. For non-uniformly hyperbolic systems the measure $\nu$ need not be
absolutely continuous with respect to Lebesgue measure. It's regularity may sometimes be quantified by
local dimension estimates. Recall that the 
pointwise local dimension of $\nu$  at $x$ is given by:
\begin{equation}\label{eq.localdim}
d_{\nu} (x):=\lim_{r\to 0}\frac{\log\nu(B(x,r))}{\log r},
\end{equation}
whenever this limit exists. Here $B(x,r)$ denotes the ball of radius $r$ centered at $x\in\XX$.
For the examples we consider the local dimension of $\nu$  exists  and is the same for $\nu$-a.e. $x\in\XX$.  We will call this value $d_{\nu}$.
For the measure $\nu$, we also need control its regularity on certain
shrinking annuli about a distinguished point $\tilde{x}$.  We say  assumption $(A2)$ holds for $\tilde{x}$ if:
\begin{enumerate}
\item[(A2)]{\bf (Regularity of $\nu$ on shrinking annuli about $\tilde{x}$).} 
There exists $\delta>0$, and $r_0>0$ such that for all $\epsilon<r<r_0$:
\begin{equation}
|\nu(B(\tilde{x},r+\epsilon))-\nu(B(\tilde{x},r))|\leq C\epsilon^{\delta}.
\end{equation}
The constant $C$ depends on $\tilde{x}$ (but not on $\delta$).
\end{enumerate}
In the statement of our results, we consider observable functions of the form 
$\phi(x)=\psi(\dist(x,\tilde{x}))$, with $\psi:\mathbb{R}^+\to\mathbb{R}$. We will assume that $\psi$ is monotonically decreasing and $\lim_{y\to 0}\psi(y)=\infty$. For this observable class, the level sets $\{\phi(x)\geq u\}$, for given $u\in\mathbb{R}$ correspond to balls in the Euclidean metric. To establish almost sure growth estimates of $M_n$, we need to know how $\nu$ scales on these level sets (as $u\to\infty$). In particular equation 
\eqref{eq.localdim}, and assumption (A2) is phrased in terms of how $\nu$ scales on small balls. For the majority of our results, we focus on the explicit cases $\psi(y)=-\log y$, and $\psi(y)=y^{-\alpha}$ (for some $\alpha>0$). However, our approach is quite general and can be adapted to other functional forms for $\psi$.  We have the following result:
\begin{thm}\label{thm.max.log}
Suppose that $(f,\XX,\nu)$ is a measure preserving system with ergodic SRB measure
$\nu$. Suppose that the local dimension $d_{\nu}$ is defined at $\tilde{x}$ and (A2) holds for $\tilde{x}$. 
Let $\phi(x)=-\log(\dist(x,\tilde{x}))$. If:
\begin{enumerate}
\item 
(A1) holds and $\Theta(n)=O(\theta^{n}_0)$ for some $\theta_0<1$, 
then
\begin{equation}\label{eq.max.exp}
\lim_{n\to\infty}\frac{M_n(x)}{\log n}=\frac{1}{d_{\nu}},\quad\textrm{a.s.}
\end{equation}
\item   (A1) holds and for $\zeta>0$, $\Theta(n)=O(n^{-\zeta})$
then there exists $\beta\in(0,1)$ such that:
\begin{equation}\label{eq.max.poly}
\frac{\beta}{d_{\nu}}\leq\liminf_{n\to\infty}\frac{M_n(x)}{\log n},\;\limsup_{n\to\infty}\frac{M_n(x)}{\log n}\leq
\frac{1}{d_{\nu}},\quad\textrm{a.s.}
\end{equation}
\end{enumerate}
\end{thm}
\begin{rmk}
For systems with superpolynomial decay of correlations, i.e. where for all $\zeta>0$ we have $\Theta(n)=O(n^{-\zeta})$, then in Case 2 of Theorem \ref{thm.max.log} the constant $\beta$ can be chosen arbitrarily close to 1. In fact, we will derive an explicit lower bound on $\beta$
in terms of $\zeta$ and $\delta$. 
\end{rmk}
The conclusions above depend essentially on having decay of correlations and a sufficiently regular ergodic measure. The recurrence statistics associated to the reference point $\tilde{x}$ do not feature in the statements. This is contrary to the case of distributional limits for $M_n$ in the context of extreme value theory. Distributional limits for linear scalings of $M_n$ in the sense of equation \eqref{eq:extreme-conv1}, i.e. the form of the limit law $G(u)$, depend on the recurrence properties of $\tilde{x}$ (e.g. periodic versus non-periodic), 
see for example \cite{HNT, GHN, FFT2, Keller, Ferguson_Pollicott}. In Section 
\ref{sec.srt} we will discuss refined estimates on the almost sure bounds of $M_n$, and these bounds will use information on the recurrence statistics associated to $\tilde{x}$, and the distributional limit behavior of $M_n$. 
If $\nu$ is absolutely continuous with respect to Lebesgue measure on $\XX$ we obtain more refined estimates.
To state the following results we use the notation $\{A_n,\mathrm{ev}\}$, ($A_n$ occurs \emph{eventually}) 
to denote the event $\bigcup_{n\geq 0}\bigcap_{k\geq n}A_k$,
when given a sequence of events $A_n\subset\XX$. Similarly we 
denote the event $\bigcap_{n\geq 0}\bigcup_{k\geq n}A_k$ by $\{A_n,\mathrm{i.o.}\}$, if $A_n$ occurs \emph{infinitely often}. We have the following result
\begin{thm}\label{thm.max.abs.cts}
Suppose that $(f,\XX,\nu)$ is measure preserving, where
$\nu$ is ergodic and absolutely continuous with respect to Lebesgue measure, and (A2) holds for $\tilde{x}$.
Let  $\phi(x)=-\log(\dist(x,\tilde{x}))$. If:
\begin{enumerate}
\item 
(A1) holds and $\Theta(n)=O(\theta^{n}_0)$ for some $\theta_0<1$, then 
there exists $\beta>0$ such that for all $\eta>1$:
\begin{equation}\label{eq.max.ev.exp}
\nu\left(x\in\XX:d^{-1}_{\nu}(\log n+\eta\log\log n)\geq M_n(x)\geq 
d^{-1}_{\nu}(\log n-\beta\log\log n),\mathrm{ev}\right)=1.
\end{equation}
\item  (A1) and $\Theta(n)=O(n^{-\zeta})$ for $\zeta>2\delta^{-1}$, then there exists $\beta>0$ such that for all 
$\eta>1$:
\begin{equation}\label{eq.max.ev.poly}
\nu\left(x\in\XX:d^{-1}_{\nu}(\log n +\eta\log\log n) \geq M_n(x)\geq 
d^{-1}_{\nu}\beta\log n,\mathrm{ev}\right)=1.
\end{equation}
\end{enumerate}
\end{thm}

In Case 2, a lower bound for the constant $\beta$ can be made explicit in terms of $\delta$ and $\zeta$.
The proof of this theorem requires the determination of (optimal) sequences $u_n$ and $v_n$ so that 
$\nu\{v_n \leq M_n\leq u_n,\,\mathrm{ev}\}=1$. From equations \eqref{eq.max.ev.exp} and \eqref{eq.max.ev.poly},
the respective almost sure convergence results of Cases 1 and 2 in Theorem \ref{thm.max.log} apply. To obtain such results
the relative sizes of the sequences $u_n$ and $v_n$ need to be determined. Indeed, if we cannot obtain sequences satisfying 
$u_n\sim v_n$, then we only obtain almost sure growth bounds on $M_n$.  The 
sequences depend on the form of the observable $\phi(x)$. 
\begin{thm}\label{thm.powerdist}
Suppose that $(f,\XX,\nu)$ is a measure preserving system with ergodic SRB measure
$\nu$. Suppose that the local dimension $d_{\nu}$ is defined at $\tilde{x}$ and (A2) holds at $\tilde{x}$. 
Let $\phi(x)=\dist(x,\tilde{x})^{-\alpha}$ for some $\alpha>0$. 
If :
\begin{enumerate}
\item (A1) holds and $\Theta(n)=O(\theta^{n}_0)$ for some $\theta_0<1$, 
then for all $n>0$, and all $\epsilon>0$:
\begin{equation}\label{eq.max.srb.exp}
\nu\{x\in\XX:\, n^{\frac{\alpha}{d_{\nu}}-\epsilon}\leq M_n(x)\leq n^{\frac{\alpha}{d_{\nu}}+\epsilon},\mathrm{ev}\}=1.
\end{equation}
\item  (A1) holds and $\Theta(n)=O(n^{-\zeta})$ for $\zeta>2\delta^{-1}$, then there exists $\beta>0$ such that
for all $\epsilon>0$:
\begin{equation}\label{eq.max.srb.poly}
\nu\{x\in\XX:\, n^{\frac{\beta\alpha}{d_{\nu}}-\epsilon}\leq M_n(x)\leq n^{\frac{\alpha}{d_{\nu}}+\epsilon},\mathrm{ev}\}=1.
\end{equation}
\end{enumerate}
\end{thm}
\begin{rmk}\label{rmk.thmpowerdist}
In the case of an absolutely continuous invariant measure $\nu$ we can further refine the estimates given in Theorem 
\ref{thm.powerdist}. See Propositions \ref{prop.mn.obs} and \ref{prop.mn.poly}.
\end{rmk}
In the case of the observable $\phi(x)=\dist(x,\tilde{x})^{-\alpha}$, the theorem above gives only almost sure bounds on 
the growth of $M_n$. Hence, we are led to the question on existence of an almost sure growth rate. The following result
shows that for certain observables there is no almost sure limit.
\begin{thm}\label{thm.nolimit}
Suppose that $(f,\XX,\nu)$ is a measure preserving system with ergodic measure
$\nu$ which is absolutely continuous with respect to Lebesgue measure, and (A2) holds at $\tilde{x}$.
Moreover suppose that we have decay of correlations in BV versus $L^1$ with rate function
$\Theta(j)$ satisfying $\sum_j\Theta(j)<\infty$. Consider the observable
 $\phi(x)=\dist(x,\tilde{x})^{-\alpha}$ for some $\alpha>0$. Then we have for any monotone sequence $u_n\to\infty$:
\begin{equation}\label{eq.mn.nolimit.thm}
\nu\left(\limsup_{n\to\infty}\frac{M_n(x)}{u_n}=0\right)=1,\;\textrm{or}\;\nu\left(
\limsup_{n\to\infty}\frac{M_n(x)}{u_n}=\infty\right)=1.
\end{equation}
\end{thm}
We remark that decay of correlations of $BV$ against $L^1$ is a strong condition. It is known to hold for certain 
expanding maps with exponential decay of correlations, see \cite{HNPV}. However, we show that the conclusion of Theorem \ref{thm.nolimit} is more widely applicable, especially for systems where quantitative recurrence statistics are known to hold for shrinking targets around generic $\tilde{x}\in\XX$. We will discuss this in Section \ref{sec.srt}. 

In Section \ref{sec.applications} we apply the main theorems given above to study the behaviour of $M_n$ for a range of dynamical
systems models of various complexity. At this stage, it is worth examining a simple case study.

\noindent {\bf The Tent Map.}

 The tent map  $f:\XX\to\XX$,  is given by $f(x)=1-|1-2x|$ on $\XX=[0,1]$. For this map  Lebesgue measure $m$ is invariant and ergodic, and condition (A1) holds with exponential decay, in fact for 
 BV versus $L^1(m)$.
 
 In this case $B(x,r)$ is just the interval $[x-r,x+r]$, and
$m(B(x,r))=2r$. It is easy to see that (A2) holds for all $\tilde{x}\in\XX$, with $\delta=1$.
For the observable $\phi(x)=-\log\dist(x,\tilde{x})$, we can apply Theorem \ref{thm.max.abs.cts}, and deduce
that for Lebesgue a.e. $x\in[0,1]$, there exists $N(x)>0$ such that for all $n\geq N$:
$$\log n+\delta\log\log n\geq M_n(x)\geq \log n-\beta\log\log n,$$
for all $\delta>1$ and some $\beta>0$ (in fact $\beta=3$ will suffice). In this case we have
$M_n/\log n\to 1$ (almost surely).

In the case of having the observable $\phi(x)=\dist(x,\tilde{x})^{-\alpha}$, the results of Section \ref{sec.general}
imply that for Lebesgue a.e. $x\in[0,1]$, there exists $N(x)>0$ such that for all $n\geq N$:   
$$n^{\alpha}(\log n)^{\alpha\delta}\geq M_n(x)\geq n^{\alpha}(\log n)^{-3\alpha}.$$
In this case we only achieve almost sure bounds. Moreover Theorem \ref{thm.nolimit} applies to this example and so
no almost sure growth rate exists.

The paper is organised as follows: in Section \ref{sec.general} we will prove the main results, and along the way, we outline a general approach for gaining almost sure growth rates for other observable types. The main approach of proof uses the theory of dynamical Borel Cantelli Lemmas in shrinking target problems, see for example \cite{GHN,HNPV}.
In Section \ref{sec.general} we give the key propositions that allow us to deduce almost sure bounds on $M_n$. 
In Section \ref{sec.applications}, we verify that our assumptions (A1) and (A2) hold for a wide class of hyperbolic systems, including the H\'enon system, and  the (2-dimensional) geometric Lorenz map. 
In Section \ref{sec.srt} we consider refinements on the bounds for $M_n$ given extra information
on the quantitative recurrence statistics and/or regularity of the invariant measure. 

\subsection{Almost sure growth of $M_n$ for i.i.d random variables}\label{sec.iid}
Suppose that $\{X_i\}$ are a sequence of i.i.d random variables with probability distribution function
$F(x)=P(X\leq x)$, and (as before) let $M_n=\max\{X_1,\ldots,X_n\}$. In the case of i.i.d random variables, almost sure growth rates of $M_n$ are fully understood, see \cite{Embrechts, Galambos}, and it is worth contrasting these results with those established here in dynamical system setting. 

Let $\overline{F}(u)=1-P\{X<u\}$, so that $F(u)\to 0$ as $u$ approaches the upper end point of the distribution
(usually taken to be $\infty$). In the following discussion we will assume that the range of $X$ is the set $[0,\infty)$.
In the case of finite upper end-point for $X$, similar results hold. In the i.i.d case, almost sure behaviour of $M_n$ is characterised in terms of three sets of sequences $\{u_n\}$, $\{v_n\}$, and $\{w_n\}$.
In \cite{Galambos} the following results are established:
\begin{enumerate}
\item {\bf (Upper bounds).} Suppose that $u_n$ is such that $\sum_n\overline{F}(u_n)<\infty$, then
$P\{M_n\geq u_n,\mathrm{i.o}\}=0$, and so $P\{M_n< u_n,\mathrm{ev}\}=1$.
\item {\bf (Lower bounds).}
Suppose that $v_n$ is such that: 
$$\sum_n\overline{F}(v_n)=\infty,\;\textrm{and}\;\sum_n\overline{F}(v_n)e^{-n\overline{F}(v_n)}<\infty,$$
then $P\{M_n\leq v_n,\mathrm{i.o}\}=0$ and so $P\{M_n> v_n,\mathrm{ev}\}=1$,
\item {\bf (Intermediate bounds).}
Suppose that $w_n$  is an increasing sequence such that: 
$$\sum_n\overline{F}(w_n)=\infty,\;\textrm{and}\;\sum_n\overline{F}(w_n)e^{-n\overline{F}(w_n)}=\infty,$$
then $P\{M_n< w_n,\mathrm{i.o}\}=1=P\{M_n>w_n,\mathrm{i.o}\}$.
\end{enumerate}
So we have the bands of fluctuation $v_n\leq w_n \leq u_n$, but it is also  possible that $v_n\sim w_n\sim u_n$ (necessary to have limiting behavior).
Clearly $P\{v_n\leq M_n\leq u_n,\,\mathrm{ev}\}=1$. The relative sizes of these sequences $u_n$ and $v_n$ depend
on the tail behaviour of the distribution function $F(u)$.

For example, we may take
$u_n$ so that $\overline{F}(u_n)<1/(n(\log n)^{\delta})$ for any $\delta>1$,
and take $v_n$ so that $\overline{F}(v_n)> \delta n^{-1}(\log\log n)$ for any 
$\delta>1$. The sequence $w_n$ would satisfy (for example):
$$\frac{\log\log n}{n}>\overline{F}(w_n)>\frac{1}{n\log n}.$$
If we are to get almost sure growth rates of $M_n$, then
we must be able to choose $u_n$ and $v_n$ (above) with the property that
$u_n/v_n\to 1$. This is not always possible.

Consider the following examples. Suppose that $X_i$ are i.i.d exponential random variables, and $F(x)=1-e^{-x}$, for $x\in[0,\infty)$. The observable $-\log d(x,\tilde{x})$ has a similar 
distribution function in the deterministic setting.

If $u_n=\log n+\delta\log\log n$ for any $\delta>1$, then $P\{M_n\leq u_n,\,\mathrm{ev}\}=1$.
If $v_n=\log n-\delta'\log\log n$ then for any $\delta'>1$  
$P(M_n\geq v_n\,\mathrm{ev}\}=1$. In this case:
$$\lim_{n\to\infty}\frac{M_n}{\log n}=1,\quad\textrm{a.s}.$$

If we take the sequence $w_n=\log n$, then $P(M_n> w_n\,\mathrm{i.o}\}=1=P(M_n< w_n\,\mathrm{i.o}\}$. 
These results are consistent with what we observe in the dynamical system setting under the assumption that 
$\phi(x)=-\log\dist(x,\tilde{x})$, for example Theorems \ref{thm.max.log} and \ref{thm.max.abs.cts}. In the i.i.d case we get improved bounds on $M_n$ (as we might expect) since our knowledge of i.i.d. Borel-Cantelli lemmas is stronger than in the deterministic setting. For example our theorems as stated do not capture all the sequences $w_n$ with the properties above. Such sequences are harder to determine, since they are prescribed in terms of fine asymptotic bounds on the tails $\overline{F}(w_n)$. 

As a second example, consider the case where $X_i$ are i.i.d and governed by a probability distribution $F(x)$
with $\overline{F}(x)\sim 1/x$ for $x\to\infty$. The observable $ \phi(x)=\dist(x,\tilde{x})^{-\alpha}$, $\alpha>0$ has a similar distribution function in the 
deterministic setting. For any sequence 
$u_n$ and $t>0$, the sums $\sum_{n}\overline{F}(u_n)$,and $\sum_{n}\overline{F}(tu_n)$ either both converge or 
both diverge. Hence
$$P\{\limsup_{n\to\infty} M_n/u_n=\infty\}=1,\;\textrm{or}\;P\{\limsup_{n\to\infty} M_n/u_n=0\}=1.$$
This result is consistent with the result stated in Theorem \ref{thm.nolimit} given the observable
$\phi(x)=\dist(x,\tilde{x})^{-\alpha}$, $\alpha>0$. See also \cite[Theorem 4.4.4]{Galambos} which gives further characterization on the distribution functions for which there is no almost sure growth rate.

Finally, there are probability distributions which have the property that:
$$P\left(\lim_{n\to\infty}(M_n-a_n)=0\right)=1,\;\textrm{for some sequence $a_n\to\infty$.}$$ An example is the Gaussian distribution, with $a_n=\sqrt{2\log n}$. Such an `additive' almost sure law for $M_n$ will be achieved if $u_n=v_n+o(1)$. For certain dynamical
systems, an observable of the form $\phi(x)=\sqrt{|\log\dist(x,\tilde{x})|}$ would give rise to an additive law. We do not explore this in detail, but the relevant sequence $a_n$ (when it exists) can be derived from Proposition~\ref{prop.Mn.bounds}.

\section{Almost sure growth rates via Borel Cantelli sequences}\label{sec.bc.hyperbolic}
In this section we discuss almost sure growth rate results for certain hyperbolic and non-uniformly hyperbolic dynamical systems, and prove the main theorems. Our arguments depend upon the interplay between the distribution function of an observable and dynamical Borel-Cantelli lemmas, which give almost sure rates of approach to  a distinguished point $\tilde{x}$.

\subsection{Borel Cantelli sequences}
We begin by recalling the classical Borel-Cantelli Lemmas for a probability space $(\XX,\mathcal{B},\nu)$,
where $\mathcal{B}$ is the $\sigma$-algebra.
\begin{enumerate}
\item If $(A_n)_{n\geq 0}$ is a sequence of measurable events in $\XX$, and $\sum_{n=0}^{\infty}\nu(A_n)<\infty$,
then $\nu\{A_n,\textrm{i.o.}\}=0$.
\item If $(A_n)_{n\geq 1}$ is a sequence of independent events in $\XX$, and $\sum_{n=0}^{\infty}\nu(A_n)=\infty$,
then $\nu\{A_n,\textrm{i.o.}\}=1$. Moreover for $\nu$-a.e. $x\in\XX$ we have 
$$\lim_{n\to\infty}\frac{S_n(x)}{E_n}=1,$$
where $S_n(x)=\sum_{j=0}^{n-1} 1_{A_j}(x)$ and $E_n=\sum_{j=0}^{n-1}\nu(A_j)$
\end{enumerate}
In general (for dependent processes), we say that a sequence $A_j$ is a strong Borel Cantelli (SBC) sequence if  $E_n\to\infty$ and
$S_n(x)/E_n\to 1$ for $\nu$-a.e. $x\in\XX$. If $E_n\to\infty$, we sat that $A_j$ is a Borel Cantelli (BC)
sequence if $S_n(x)\to\infty$ for $\nu$-a.e. $x\in\XX$. To study almost sure growth rates of $M_n$, we will use refined estimates on the relative growth of $S_n(x)$ and $E_n$. In particular we have the following:
\begin{prop}\label{prop.bc2}
Suppose that $(f,\XX,\nu)$ is a measure preserving system with ergodic
SRB measure $\nu$,  that (A2) holds for $\tilde{x}\in \XX$ and the local dimension $d_{\nu}$ exists at $\tilde{x}$. 
\begin{enumerate}
\item If (A1) holds with $\Theta(n)=O(n^{-\zeta})$ for some $\zeta>0$, and
$B_n=B(\tilde{x},r_n)$ is a sequence of decreasing balls about $\tilde{x}$ with 
$\nu(B_n)=n^{-\beta}$ for some $\beta\in(0,1)$ then for $\nu$-a.e. $x\in\XX$:
\begin{equation}\label{eq:sprindzuk.weak1}
  S_n(x) =E_n  +
  O (E^{\beta'}_{n}),\;\textrm{for any}\;\beta'>\frac{2\delta^{-1}+\beta+\zeta(1-\beta)}{2\zeta(1-\beta)},
  \end{equation} 
  where $S_n(x)=\sum_{k\leq n}1_{B_k}(f^{k}(x))$ and  $E_n=\sum_{k\leq n} \nu(B_k)$ 
\item 
If (A1) holds with $\Theta(n)=O(\theta^{n}_0)$ for some $\theta_0<1$ and 
$B_n=B(\tilde{x},r_n)$ is a sequence of decreasing balls about $\tilde{x}$ with 
$\nu(B_n)=(\log n)^{\beta}/n$ then
for $\nu$-a.e. $x\in\XX$: 
\begin{equation}\label{eq:sprindzuk.weak2}
  S_n(x) =E_n  +
  O (E^{\beta'}_n),\;\textrm{for any}\;\beta'>\frac{2+\beta}{2(1+\beta)}.
  \end{equation} 
\end{enumerate}
\end{prop}
\begin{rmk}
For the error term $O(E^{\beta'}_n)$ to be useful,  we require $\beta'<1$. In Case 1 
we require $\beta, \delta$ and $\zeta$ to satisfy
$$2\delta^{-1}<\zeta-\zeta\beta-\beta.$$
If $\zeta>2\delta^{-1}$, then the ball radii must be chosen so that $\nu(B_n)=n^{-\beta}$, with
$\beta<(\zeta-\delta^{-1})/(\zeta-1).$ In Case 2 we just require $\beta>0$. For systems with superpolynomial decay 
of correlations we can take $\beta$ arbitrarily close to 1. It is clear that if $\beta'<1$ we have 
$S_n (x)/E_n\to 1$ as $n\to\infty$ for $\nu$ a.e.\ $x\in \XX$. 
\end{rmk}
\begin{rmk}
The error estimates in Proposition \ref{prop.bc2} cannot be usefully applied
to balls $B_n$ with $\nu(B_n)=O(1/n)$. In Section~\ref{sec.srt}, we extend these estimates to more general
nested balls under assumptions on the quantitative recurrence statistics. 
\end{rmk}

\subsection{Almost sure growth rates for $M_n$: the general approach}\label{sec.general}
Using Proposition \ref{prop.bc2} we deduce almost sure bounds on the growth of $M_n$. 
We consider first the case of having exponential decay of correlations and having an absolutely continuous
invariant measure. We obtain the strongest bound in this case, see Proposition \ref{prop.Mn.bounds}. We then focus
on particular observables, Proposition \ref{prop.mn.obs}, and then consider the case of having polynomial decay
of correlations: Proposition \ref{prop.mn.poly}. In Section \ref{sec.proof.thm} we treat the case of having general SRB measures and conclude the proof of the main theorems as stated in Section \ref{sec.results}.

\subsubsection{Almost sure growth of $M_n$ for systems with exponential decay of correlations}\label{sec.general1} We discuss a  general approach to  finding upper and lower bounds on the growth rate of $M_n$ under the assumption that $\phi(x)=\psi(\dist(x,\tilde{x}))$, where $\psi:\mathbb{R}^{+}\to\mathbb{R}$ is a monotonically decreasing function
with $\psi(y)\to\infty$ as $y\to 0$. 
In the following we consider the case for dimension $d=1$. The arguments generalize easily to higher dimensions. 

\begin{prop}\label{prop.Mn.bounds}
Suppose that $(f,\XX,\nu)$ is a measure preserving system with an absolutely continuous ergodic invariant
measure $\nu$. 
Consider the observable $\phi(x)=\psi(\dist(x,\tilde{x}))$. Suppose that Condition (A1) holds with 
$\Theta(n)=O(\theta^{n}_0)$, and condition (A2) holds for $\tilde{x}$. Then for $\nu$-a.e. $x\in\XX$, and all $\delta_1>1$: 
\begin{equation}
\nu\left\{\psi\left(\frac{(\log n)^{3}}{n}\right)\leq M_n\leq
\psi\left(\frac{1}{n(\log n)^{\delta_1}}\right),\,\mathrm{ev}\right\}=1.
\end{equation}
\end{prop}
\begin{proof}
In light of Proposition \ref{prop.bc2} it suffices to consider a sequence of balls $B_n$ satisfying
$\nu(B_n)=n^{-1}(\log n)^{\beta}$ for some $\beta>0$. This implies that their radii scale as 
$r_n\sim Cn^{-1}(\log n)^{\beta}$ for some $C>0$. Here, the constant $C$ depends on the density of $\nu$ at $\tilde{x}$. In dimension $d$ we would choose balls of radii $r_n^{1/d}$ but for the sake of exposition we take $d=1$ from now on.
 
Given $x\in\XX$ and $n\geq 1$, define $n_l(x):=\max\{k\leq n: T^k(x)\in B_k\}.$ 
Recalling that $S_n(x)=\sum_{k\leq n}1_{B_k}(x)$, 
$S_{n_l}(x)=S_n(x)$, and $M_{n}(x)\geq M_{n_l}(x).$ 

By Proposition \ref{prop.bc2} we have for any $\beta'>\frac{2+\beta}{2(1+\beta)}$ and for
$\nu$-a.e. $x\in\XX$:
 \begin{equation}\label{eq.p-bd}
E_n-E_{n_l}\leq O((E_n)^{\beta'}),
\end{equation}
Since $\nu(B_n)=n^{-1}(\log n)^{\beta}$ we have the asymptotic:
$E_n= C_{\beta}(\log n)^{1+\beta}+o(1)$. Hence equation \eqref{eq.p-bd} implies that
$$(\log n_{l})^{1+\beta}\geq (\log n)^{1+\beta}-O((\log n)^{\beta'(1+\beta)}),$$
Inverting for $n_l$ gives:
$$n_{l}\geq \exp\left\{(\log n)\left(1-c(\log n)^{(\beta'-1)(\beta+1)}\right)^{\frac{1}{\beta+1}}\right\},$$
for some $c>0$. An asymptotic expansion, implies we can write $n_{l}\geq\exp\{\log n-\ell(n)\}$, where $\ell(n)>0$, and
$$\ell(n)=O\left((\log n)^{(\beta'-1)(\beta+1)+1}\right).$$
For small values of  $\beta$, we have $\ell(n)=O((\log n)^{\sigma})$ for some $\sigma>0$ and in this case 
$\ell(n)\to\infty$. If 
$\beta'<1$ and $\beta>\beta'/(1-\beta')$, then $\ell(n)=o(1)$, and
$n_l\sim n$ as $n\to\infty$. For this choice of $\beta$ we have for all $\epsilon>0$:
$$M_n(x)\geq M_{n_l}(x)\geq \psi(r_{n_l})\sim\psi\left(C\frac{(\log n_l)^{\beta}}{n_l}\right)
\geq \psi\left(\frac{(\log n)^{\beta+\epsilon}}{n}\right).$$
We now find a lower bound on the value of $\beta$. We require the following to hold simultaneously:
$\beta>\beta'/(1-\beta')$, and $1>\beta'>\frac{2+\beta}{2(1+\beta)}$. For any 
$\beta>5/2$, then these inequalities are simultaneously satisfied (with $\beta'$ chosen accordingly). 

To get an upper bound on the growth of $M_n$, 
we choose a sequence $B_k$ such that  $\sum_{k}\nu(B_k)<\infty.$ Since $\nu$ is absolutely continuous with respect to Lebesgue measure, for $\delta_1>1$
a sequence of balls of radius $r_n\sim\frac{1}{n(\log n)^{\delta_1}}$ will do.
Putting these estimates together we achieve for any $\delta_1>1$:
\begin{equation}\label{eq.mn-as.obs}
\nu\left\{\psi\left(\frac{(\log n)^{3}}{n}\right)\leq M_n\leq
\psi\left(\frac{1}{n(\log n)^{\delta_1}}\right),\,\mathrm{eventually}\right\}=1.
\end{equation}
\end{proof}

\subsubsection{Almost sure growth of $M_n$ for specific observables}
The bounds on $M_n$ depend on the functional form of $\psi(y)$. If $\psi(y)$ is a power
law function of the form $\psi(y)=y^{-\alpha}$ for $\alpha>0$, then we achieve almost sure bounds on the growth
of $M_n$, but do not establish an almost sure growth rate. The following proposition deals with
exponential decay of correlations and an absolutely continuous invariant measure.
\begin{prop}\label{prop.mn.obs}
Suppose that $(f,\XX,\nu)$ is a measure preserving system with an absolutely continuous ergodic invariant
measure $\nu$. Suppose that Condition (A1) holds with $\Theta(n)=O(\theta^{n}_0)$, and condition (A2) holds for 
$\tilde{x}$. 
Then:
\begin{enumerate}
\item if $\phi(x)=-\log\dist(x,\tilde{x})$ we have
$$\lim_{n\to\infty}\frac{M_n(x)}{\log n}=\frac{1}{d}.\quad\textrm{a.s.}$$
\item if $\phi(x)=\dist(x,\tilde{x})^{-\alpha}$ for some $\alpha>0$,
for any $\delta>0$ we have
$$\nu\{x\in\XX:\, n^{\frac{\alpha}{d}}(\log n)^{-\frac{3\alpha}{d}}\leq M_n(x)\leq n^{\frac{\alpha}{d}}(\log n)^{\frac{\alpha\delta}{d}},\mathrm{eventually}\}=1.$$
\end{enumerate}
\end{prop}
Proposition \ref{prop.mn.obs} gives an indication on how the upper and lower bounds on the growth rate of $M_n$
depend on the form of the observable. In Case 2, we cannot deduce an almost sure growth rate for $M_n$.
\begin{proof}
Case 1: to find the upper bound for $M_n$, we consider a sequence
of balls $B_n$  such  that $\sum_n\nu(B_n)<\infty$. Following the approach discussed above, we take balls of radius 
$r_n\sim C\frac{1}{n(\log n)^{\delta}}$, (for $\delta>1$), so that $\nu(B_n)=\frac{1}{n(\log n)^{\delta}}$.
Then  $$ \psi\left(\frac{1}{n(\log n)^{\delta}}\right) \sim \log n+\delta\log\log n,$$
and eventually $M_n\leq  \log n+\delta\log\log n$.
To get the lower bound, it suffices to consider a sequence of balls $B_n$ of radius 
$r_n\sim Cn^{-1}(\log n)^{\beta}$ for some $C,\beta>0$ so that $\nu(B_n)=n^{-1}(\log n)^{\beta}$.
From Proposition \ref{prop.Mn.bounds} we obtain for any $\beta\geq 2.5$ and $\epsilon>0$:
$$M_n(x)\geq \psi\left(\frac{(\log n)^{\beta+\epsilon}}{n}\right)=\log n-(\beta+\epsilon)\log\log n.$$
Hence, almost surely we have $M_n/\log n\to 1$. 

Case 2: to find the upper bound for $M_n$ is suffices to take $B_n$ so that $\sum_n\nu(B_n)<\infty$. 
If we take balls of radius $r_n\sim C\frac{1}{n(\log n)^{\delta}}$, ($\delta>1$), we obtain:
$$M_n \leq\psi\left(\frac{1}{n(\log n)^{\delta}}\right)\leq n^{\alpha}(\log n)^{\alpha\delta}.$$
To get the lower bound, again consider a sequence of balls $B_n$ of radius 
$r_n\sim Cn^{-1}(\log n)^{\beta}$ for some $C,\beta>0$. We obtain for any $\beta\geq 2.5$, and any $\epsilon_1>0$:
$$M_n(x)\geq \psi\left(\frac{(\log n)^{\beta+\epsilon_1}}{n}\right),$$
and this give the required lower bound.

If $\XX$ has dimension $d>1$ the following statements may be proved in exactly the same way, by taking instead
balls of radii $r^{1/d}_n$.

\end{proof} 

\subsubsection{Almost sure growth of $M_n$ for systems with polynomial decay of correlations}
We now consider the case of polynomial decay of correlations. Again, for simplicity we take
dimension $d=1$. We consider the case of having an absolutely continuous invariant measure, and a general 
observable $\phi(x)=\psi(\dist(x,\tilde{x}))$.
\begin{prop}\label{prop.mn.poly}
Suppose that $(f,\XX,\nu)$ is a measure preserving system with an absolutely continuous ergodic invariant
measure $\nu$, and $\phi(x)=\psi(\dist(x,\tilde{x}))$. Suppose that Condition (A1) holds with $\Theta(n)=O(n^{-\zeta})$ for some $\zeta>0$, 
and Condition (A2) holds for $\tilde{x}$. Then for all $\delta_1>0$,
\begin{equation}\label{eq.mn.poly}
\nu\left\{  \psi\left(\frac{1}{n^{\beta}}\right)  \leq M_n\leq
\psi\left(\frac{1}{n(\log n)^{\delta_1}}\right),\,\mathrm{ev}\right\}=1,\;
\textrm{for any}\; \beta<\frac{\zeta-2\delta^{-1}}{1+\zeta}. 
\end{equation}
\end{prop}
\begin{rmk}
In Proposition \ref{prop.mn.poly} we require $\beta>0$ for the estimate in equation \eqref{eq.mn.poly} to be useful. This will hold provided $\frac{\zeta-2\delta^{-1}}{1+\zeta}<1$. For systems with superpolynomial decay of correlations we can take $\beta$ arbitrarily close to 1.
\end{rmk}
\begin{rmk} For the explicit representations $\psi(y)=-\log y$, or $\psi(y)=y^{-\alpha}$, we can
derive similar statements to those stated in Proposition \ref{prop.mn.obs} (but with slightly weaker bounds). 
\end{rmk}
\begin{proof}
We follow the proof of Proposition \ref{prop.Mn.bounds}, but instead take a sequence of balls $B_n$ with
$\nu(B_n)=n^{-\beta}$ for some $\beta\in(0,1)$.
As before 
$S_{n_l}(x)=S_n(x)$, and $M_{n}(x)\geq M_{n_l}(x)$, (with $n_l$ as defined in Proposition \ref{prop.Mn.bounds}). 
By Proposition \ref{prop.bc2} we have
 \begin{equation}\label{eq.p-bd2}
E_n-E_{n_l}\leq O((E_n)^{\beta'}),
\end{equation}
for any 
\begin{equation}\label{eq.beta.bound}
\beta'>\frac{2\delta^{-1}+\beta+\zeta(1-\beta)}{2\zeta(1-\beta)}.
\end{equation}
To get a lower bound on the growth rate of $M_n$, we choose balls $B_n$ with $\nu(B_n)=n^{-\beta}$ for some $\beta>0$, 
and hence $r_n\sim Cn^{-\beta}$ for some $C>0$.
It follows that $E_n\sim Cn^{1-\beta}$, and 
using equation \eqref{eq.p-bd2} we obtain
$$n_{l}^{1-\beta}\geq n^{1-\beta}-O(n^{\beta'(1-\beta)}).$$
If $\beta'<1$ then $n_l\geq n(1+o(1))$ as $n\to\infty$. Hence we obtain
$$M_n(x)\geq M_{n_l}(x)\geq \psi(r_{n_l})\geq \psi\left(\frac{C}{n^{\beta/d}}\right),$$
again for some $C>0$. To compute the optimal value of $\beta$, we just require $\beta'<1$. From equation 
\eqref{eq.beta.bound} this will be valid provided: $$\beta<\frac{\zeta-2\delta^{-1}}{1+\zeta}.$$
The upper bound for $M_n$ can be achieved by setting $r_n\sim C/[n(\log n)^{\delta_1}]$, (for arbitrary $\delta_1>1$), and using the first Borel-Cantelli Lemma.
\end{proof}

\subsubsection{Proof of main results}\label{sec.proof.thm}
Using Propositions \ref{prop.Mn.bounds}, \ref{prop.mn.obs} and \ref{prop.mn.poly} we can now complete the proof
of the main theorems. The proof of Theorem \ref{thm.nolimit} is more elaborate and hence is proved in Section 
\ref{sec.proof.nolimit}.

\paragraph{Proof of Theorem \ref{thm.max.log}.} To prove this theorem, we follow the proof
of Proposition~\ref{prop.Mn.bounds}. In the case of SRB measures we cannot make precise asymptotic statements on the radii $r_n$ of the balls $B_n$. This is due
to the fluctuation properties of $\nu$ on balls $B_n$ when $\nu(B_n)\to 0$. When the local dimension $d_{\nu}$ at 
$\tilde{x}$ exists we have the following: for all $\epsilon>0$, there exists $N$, such that for all $n\geq N$:  
$$r_n\in [\nu(B_n)^{1/d_{\nu}+\epsilon},\nu(B_n)^{1/d_{\nu}-\epsilon}].$$ 

Given this, we can now prove Case 1 of Theorem \ref{thm.max.log}. Following the proof of Proposition \ref{prop.mn.obs}, 
it suffices to take the sequence of
balls $B_n:=B(\tilde{x},r_n)$ with $\nu(B_n)=\frac{1}{n(\log n)^{\delta}}$, ($\delta>1$). Using local dimension
estimates, we obtain for all $\epsilon>0$:
$$M_n \leq\psi(r_n)\leq \left(d^{-1}_{\nu}+\epsilon\right)\log n,$$
where in the above estimate the $\delta\log\log n$ contribution has been subsumed.  
To get the lower bound, again consider a sequence of balls $B(\tilde{x},r_n)$ with 
$\nu(B_n)=n^{-1}(\log n)^{\beta}$, and apply Proposition \ref{prop.Mn.bounds} for $\beta\geq 3$. 
Then for all $\epsilon>0$:
$$M_n(x)\geq \psi (r_{n_l})\geq \left(d^{-1}_{\nu}-\epsilon\right)\log n.$$
(noting that the $\beta\log\log n_l$ contribution is asymptotically insignificant).
As before, $n_l(x):=\max\{k\leq n: T^k(x)\in B_k\}$, and we have $n_l\sim n$ for $\beta\geq 3$.  
The almost sure bounds on $M_n$ follow for Case 1.
 
Now consider Case 2. We repeat the proof of Proposition~\ref{prop.mn.poly}. The upper bound
on $M_n$ is identical to that achieved in Case 1 above, and hence almost surely
$$\limsup_{n\to\infty}\frac{M_n(x)}{\log n}\leq \frac{1}{d_{\nu}}.$$
For the lower bound, we take balls $B(\tilde{x},r_n)$ with
$\nu(B_n)=n^{-\beta}$, and $\beta$ satisfying
$$\beta<\frac{\zeta-2\delta^{-1}}{1+\zeta},$$
so that Proposition \ref{prop.Mn.bounds} applies with $n_l\sim n$.
Using the fact that $M_n\geq \psi(r_{n_l})$, we obtain for all $\epsilon>0$:
$$\liminf_{n\to\infty}\frac{M_n(x)}{\log n}\geq \frac{\beta}{d_{\nu}}-\epsilon,$$
and the conclusion follows.

\paragraph{Proof of Theorem \ref{thm.max.abs.cts}.} The proof of this result is straightforward.
In Case 1, we just apply Proposition \ref{prop.mn.obs}, and hence equation \eqref{eq.max.ev.exp}
follows for $\beta=3$. For Case 2, the result immediately follows from Proposition \ref{prop.mn.poly}

\paragraph{Proof of Theorem \ref{thm.powerdist}.} The proof of this result is achieved by repeating the method of proof
of Theorem \ref{thm.max.log} to the case of the observable function $\phi(x)=-\dist(x,\tilde{x})^{-\alpha}$,
and the result follows. 

\subsubsection{Non-existence of an almost sure growth rate for $M_n$}\label{sec.proof.nolimit}
In this section we prove Theorem \ref{thm.nolimit}, and show that for certain observables we have no almost sure scaling
sequence for $M_n$. We begin with the following lemma.
\begin{lemma}\label{lem.nolimit1}
Suppose that $(f,\XX,\nu)$ satisfies the assumption of Theorem \ref{thm.nolimit}, and suppose that
for any monotone sequence $r_n\to 0$ we have $\sum_n \nu(B(\tilde{x},r_n))=\infty$. Consider
the observable function $\phi(x)=\psi(\dist(x,\tilde{x}))$,
with $\psi(y)$ monotonically decreasing, and $\lim_{y\to 0}\psi(y)=\infty$. Then we have:
$$\nu\{M_n\geq \psi(r_n),\mathrm{i.o.}\}=1.$$
\end{lemma}
\begin{proof}
For systems $(f,\XX,\nu)$ that satisfy decay of correlations of BV versus $L^1$ it is shown that 
$\nu\{X_n\geq\psi(r_n),\mathrm{i.o.}\}=1$, where $B(\tilde{x},r_n)$ is any sequence of balls satisfying 
$\sum_n\nu(B_n)=\infty$, see \cite{HNPV,Kim}. From \cite[Theorem 4.2.1]{Galambos}, it is shown that for any
monotone sequence $u_n\to\infty$ we have
$$\nu\{M_n\geq u_n,\mathrm{i.o.}\}=\nu\{X_n\geq u_n,\mathrm{i.o}\}.$$
Hence the result follows by taking the sequence $u_n=\psi(r_n)$.
\end{proof}
Now consider an arbitrary sequence $u_n\to\infty$.
The event $\{X_n\geq u_n\}$ corresponds to the event $B(\tilde{x},r_n)$ with $u_n=\psi(r_n)=r^{-\alpha}_n$. Hence there
is a one-to-one correspondence between the sequences $u_n\to\infty$ and $r_n\to 0$. We consider two cases, namely
upper and lower sequences $r^{U}_n$, $r^{L}_n$ such that
$\sum_n \nu(B(\tilde{x},r^{U}_n))<\infty$ and $\sum_n\nu(B(\tilde{x},r^{L}_n))=\infty$.
In the former case, the first Borel Cantelli lemma implies that $\nu\{X_n\geq u_n,\mathrm{i.o}\}=0$,
where $u_n:=\psi(r^{U}_n)$. For any $t>0$, consider now the sequence of balls $B(\tilde{x},tr^{U}_n)$.
By absolute continuity of $\nu$, we have
$$\sum_n\nu(B(\tilde{x},t r^{U}_n))\sim
\sum_n \rho(\tilde{x})tr^{U}_n\sim t \sum_n\nu(B(\tilde{x},r^{U}_n))<\infty.$$ 
It follows that for all $t>0$:
$$\nu\{X_n\geq t^{-\alpha}\psi(r^{U}_n),\mathrm{i.o}\}=0\;\implies\;\nu\{X_n\leq t^{-\alpha}\psi(r^{U}_n),
\mathrm{eventually}\}=1,$$
Thus if $u_n=\psi(r_n)$ is such that the corresponding measures $\nu(B(\tilde{x},r_n)$) are summable, then 
$$\nu\left(\limsup_{n\to\infty}\frac{M_n(x)}{u_n}=0\right)=1.$$

Now suppose that $\sum_n \nu(B(\tilde{x},r^{L}_n))=\infty$ for some sequence $r^{L}_n.$ Then by Lemma \ref{lem.nolimit1}
it follows that $\nu\{M_n\geq u_n,\mathrm{i.o}\}=1$, 
where $u_n=\psi(r^{L}_n)$. Again by scaling the radii $r^{L}_n$ by a multiplying factor $t>0$, we find that the non-summability
properties of $\nu(B_n)$ are preserved. Hence for any $t>0$ we have
$\nu\{X_n\geq t^{-\alpha}\psi(r^{L}_n),\mathrm{i.o}\}=1,$
and so
$\nu\{M_n\geq t^{-\alpha}\psi(r^{L}_n),\mathrm{i.o}\}=1.$
Therefore for all such sequences $u_n=\psi(r_n)$ with $\sum_n\nu(B_n)=\infty$ we have:
$$\nu\left(\limsup_{n\to\infty}\frac{M_n(x)}{u_n}=\infty\right)=1.$$
This completes the proof of Theorem \ref{thm.nolimit}
\begin{rmk}
In the proof of Theorem \ref{thm.nolimit} we used the fact that $\nu$ was absolutely continuous. In the proof, the main property required 
is that we have $\nu(B(\tilde{x},tr))\sim t\nu(B(\tilde{x},r))$ for all $t>0$ and for all $r$ sufficiently small.
We also require the conclusion of Lemma \ref{lem.nolimit1} to apply. In Section \ref{sec.srt} we will see that the 
conclusion of Theorem \ref{thm.nolimit} applies to a broader class of systems.
\end{rmk}

\section{Application of results: non-uniformly hyperbolic systems}\label{sec.applications}
In this section we discuss a range of dynamical systems to which our results apply. These include diffeomorphisms modelled by Young towers, billiard maps, Lozi maps,
geometric Lorenz maps and intermittent type maps. From a point of view of application of results, this means checking the assumptions e.g. (A1), (A2) apply to the particular system at hand. Often the local dimension estimates apply only almost surely, and similarly with respect to assumption (A2). 

\subsubsection*{Rank one Young towers}
We consider diffeomorphisms modelled by Young towers with exponential tails. See~\cite{Young} for details of Young Towers.

\begin{thm}\label{thm.tower}
Suppose that $f:\XX\to \XX$ is a $C^2$ diffeomorphism modelled by a Young tower with SRB measure $\nu$ and 
$\nu\{R>n\}=O(\theta^{n}_0)$ for some $\theta_0<1$. For $\nu$ a.e. $\tilde{x}\in \XX$, if $\phi(x)=-\log(\dist(x,\tilde{x}))$
then we have
\begin{equation}\label{eq.max.as}
\lim_{n\to\infty}\frac{M_n(x)}{\log n}=\frac{1}{d_{\nu}},\quad\textrm{a.s.}
\end{equation}
\end{thm}
To prove this result we check (A1), together with a slightly weaker version of (A2) (which holds for $\nu$ a.e. 
$\tilde{x}$). This is a technical argument and we give the proof in Section \ref{sec.appendix}. A natural application is the H\'enon map 
$$f(x,y)=(1-ax^2+y,bx),$$ where $(a,b)\in\mathbb{R}^2$. It is shown that for a positive measure set of these parameters,
$(f,\XX,\nu)$ admits a Young tower with exponential return time asymptotics, see \cite{BY}. 
%We state the following:  
%\begin{cor}\label{cor.henon}
%Suppose that $f:\XX\to \XX$ is the H\'enon map. If $\phi(x)=-\log\dist(x,\tilde{x})$, then
%for a positive measure set of parameters $(a,b)\in\mathbb{R}^2$ we have:
%\begin{equation}
%\lim_{n\to\infty}\frac{M_n(x)}{\log n}=\frac{1}{d_{\nu}(\tilde{x})},\;\textrm{a.s.},
%\end{equation}
%where $d_{\nu}$ is the local dimension of $\nu$ at $\tilde{x}.$
%\end{cor}

\subsubsection*{Lozi maps}
Lozi maps are generally given by the equation $T(x, y) = (1+by-a|x|, x).$ The Lozi map can be modelled by a Young tower \cite{Young} with exponential tails for the return time function $R$ and hence (A1) applies. However it has discontinuities in its derivative (at $x=0$), and hence we can not immediately apply Theorem \ref{thm.tower}. However it is shown in \cite{HNPV,GHN} that (A2) holds, and so Case 1 of Theorem \ref{thm.max.log} applies to this family of maps for $\nu$ almost every $\tilde{x}$.

\subsubsection*{Hyperbolic billiards}
For a review of billiard maps and their statistical properties, see \cite{BSC2, Chernov}. Certain billiard transformations can be modelled by a Young tower with exponential return times, see \cite{Young}. However, such
transformations are not captured by Theorem \ref{thm.tower} due to the presence of singularities in their derivatives.
Strong Borel Cantelli lemmas have been established for such transformations in \cite{HNPV}, and 
(A1) holds.  (A2) holds for all $\tilde{x}$ as the invariant measure is a volume measure. Hence Case 1 of Theorem \ref{thm.max.log} applies to this family of maps for all $\tilde{x}\in \XX$.

\subsubsection*{Lorenz maps}
We consider the family of two-dimensional Poincar\'e return maps associated to 
the geometric Lorenz flow, \cite{GW}. For such a system $(f,\XX,\nu)$, the set
$\XX$ is a compact planar section in $\mathbb{R}^2$ (transverse to the Lorenz flow), and $\nu$ is an 
ergodic SRB measure. The hyperbolic properties of these maps are described in \cite{GP}, where it is shown 
that condition (A1) holds with $\Theta(n)=O(\theta^{n}_0)$ for some $\theta_0<1$. In 
\cite{Licheng} extreme distributional limits are derived, and condition (A2) is shown to holds for all $\tilde{x}\in\XX$.  Hence Case 1 of Theorem \ref{thm.max.log} applies to the family of Lorenz maps for all $\tilde{x}\in \XX$.

\subsubsection*{Intermittent maps}
As an example, one class of intermittent type maps $f:[0,1]\to[0,1]$ take the form
\begin{equation}\label{eq.intermittent}
f(x) =
\begin{cases}
x(1 + 2^{\alpha} x^{\alpha}) &\text{if } 0 \le x < \frac{1}{2};\\
2x-1 &\text{if } \frac{1}{2} \le x \le 1,
\end{cases}
\end{equation}
where $\alpha\in(0,1)$. Strong Borel Cantelli results are discussed in \cite{HNPV}, and see also
\cite{Gouezel, Kim}. It is shown that (A1) applies, with $\Theta(n)=O\left(n^{1-\frac{1}{\alpha}}\right)$, 
In particular $\nu$ is absolutely continuous with respect to Lebesgue measure with a density in $L^{p}(\mathrm{Leb})$ for some $p>1$, and hence (A2) holds for some $\delta>0$. Hence Case 2 of Theorem \ref{thm.max.abs.cts} applies to this family of maps for all $\tilde{x}\in \XX$.
%However, there is some degree of sharpness to the results: it is shown in \cite{Kim} that for the sequence of balls 
%$B_n=[0,r_n)$ with $\nu(B_n)\sim 1/n$, we have $\nu\{f^nx\in B_n,\mathrm{i.o.}\}=0$. 

\section{Refined estimates on almost sure growth rates of $M_n$ for non-uniformly hyperbolic systems.}\label{sec.srt}
In the previous sections, we assumed just decay of correlations (A1), and regularity of the invariant measure (A2).
We make the following assumption (A3) concerning so called \emph{short return times} to the distinguished point $\tilde{x}$, see for example \cite{HNPV}.
This allows us to obtain stronger Borel-Cantelli results for nested balls based at $\tilde{x}$, and hence we more closely approximate the i.i.d. case. The short return times assumption has been
shown to hold for $\nu$ a.e. $\tilde{x}\in \XX$ for a variety of non-uniformly hyperbolic systems.

\begin{enumerate}
\item[(A3)]{\bf (Lack of Short Return Times  to $\tilde{x}$).}
There exists $\epsilon_1>0$ such that if $\{B(x,r_n)\}$ is a sequence of balls with 
$\sum_n\nu(B(x,r_n))=\infty$ and $\nu(B(x,r_n))=O(n^{-1+\epsilon_1})$, then either:
\begin{enumerate}
\item[(A3a)] There exists $\alpha>0$
and $\gamma>0$ such that for $\nu$-a.e. $x\in\XX$:
\begin{equation}\label{eq.short1}
\nu\left(B(x,u_n)\cap f^{-k}B(x,u_n)\right)\leq \nu\left(B(x,u_n)\right)^{1+\alpha},
\end{equation}
for all $k=1,\ldots,g(n)$ with $g(n)=n^{\gamma}$; or
\item[(A3b)] There exists $\alpha>0$ such that equation \eqref{eq.short1} holds  but with $g(n)=(\log n)^{\gamma}$ 
for some $\gamma>1$.
\end{enumerate}
\end{enumerate}
Conditions (A3a) and (A3b) (collectively denoted as (A3)) give  restrictions on the recurrence properties of $\tilde{x}\in\XX$. For a broad class of systems, condition (A3) is shown to hold for $\nu$-a.e. $\tilde{x}\in\XX$. The exceptional set of points
includes periodic points (for example). However we believe that the almost sure behavior of $M_n$ will not depend upon the validity of the 
lack of short returns, rather it is just an artefact of our method of proof.
%For systems with exponential decay of correlations condition
%(A3b) is known to hold, see \cite{HNPV, HRS} (and references therein). Condition (A3b) is shown to hold for certain non-uniformly expanding systems with polynomial decay of correlations.
The following proposition establishes strong-Borel Cantelli results for general dynamical systems satisfying hypotheses (A1)-(A3).
\begin{prop}\label{prop.bc.srt}
Suppose that $(f,\XX,\nu)$ is a measure preserving system with ergodic
SRB measure $\nu$, and that the local dimension $d_{\nu}$ exists
at $\tilde{x}\in\XX$ and (A2) holds for $\tilde{x}$. We have the following cases.
\begin{enumerate}
\item 
Suppose that $\Theta(n)=O(\theta^{n}_0)$ for some $\theta$, 
and (A1) holds. Moreover, suppose that
$B_n=B(\tilde{x},r_n)$ is a sequence of decreasing balls about $\tilde{x}$ with 
$\limsup_n \nu(B_n)(\log n)^{\frac{\sigma}{\zeta}}<\infty$ 
for some $\sigma>1$, and (A3b) holds at $\tilde{x}$ for $\gamma>1$. If $S_n(x)=\sum_{k\leq n}1_{B_k}(f^{k}(x))$, $E_n=\sum_{k\leq n} \nu(B_k)$ then 
\begin{equation}\label{eq.srt.exp}
S_n(x) =E_n  +
  O (E^{1/2}_{n} \log^{3/2+\epsilon}E_{n},
  ),  
  \end{equation}
\item Suppose that $\Theta(n)=O(n^{-\zeta})$ and (A1) 
holds together with (A3a) at $\tilde{x}$, with $\gamma>(2\delta^{-1}+1)/(\zeta-1)$. 
Moreover suppose that
$B_n=B(\tilde{x},r_n)$ is a sequence of decreasing balls about $\tilde{x}$ with 
$\limsup_n \nu(B_n)n^{\sigma/\alpha}<\infty$, for some
$$\sigma>\frac{2\delta^{-1}+1}{\zeta-1}.$$ 
If $E_n\to\infty$ then 
\begin{equation}\label{eq.srt.poly}
S_n(x) =E_n  +
  O (E^{1/2}_{n} \log^{3/2+\epsilon}E_{n},
  ),  
\end{equation}
\end{enumerate}
\end{prop}
\begin{rmk}\label{rmk.srt}
In Case 2, the estimate is of no use  if $(2\delta^{-1}+1)/[\alpha(\zeta-1)]\geq 1.$
\end{rmk}
Proposition \ref{prop.bc.srt} has the immediate corollaries:
\begin{cor}\label{cor.srt.io}
Suppose that $(f,\XX,\nu)$ satisfies the assumptions of Proposition \ref{prop.bc.srt}, with either case applying but in Case 2 we require $\frac{\sigma}{\alpha}<1$.
Suppose that for some monotone sequence $r_n\to 0$, $\sum_n \nu(B(\tilde{x},r_n))=\infty$. 
Then for any observable function $\phi(x)=\psi(\dist(x,\tilde{x}))$ with $\psi(y)$ monotone increasing (as $y\to 0$):
$$\nu\{X_n\geq \psi(r_n),\mathrm{i.o.}\}=1.$$
\end{cor}
\begin{proof}
From Proposition \ref{prop.bc.srt}, the corollary is  true for balls $B_n$ with 
$\nu(B_n)=O(n^{-1+\tilde\epsilon})$ for some $\tilde\epsilon>0$. The existence of $\tilde\epsilon$ follows
by (A3), and the assumptions placed on $\nu(B_n)$ in the statement of the Proposition.
We now extend the Borel-Cantelli property to all sequences $r_n$, with $\sum_n \nu(B(\tilde{x},r_n))=\infty$.
Let $\tilde{r}_n$ be a sequence such that $\tilde{r}_n=r_n$ if $\nu(B(\tilde{x},r_n))\leq n^{(-1+\epsilon_1)}$, and let
$\tilde{r_n}$ be such that $\nu(B(\tilde{x},\tilde{r}_n))=n^{(-1+\epsilon_1)}$ otherwise. The corresponding balls 
$B(\tilde{x},\tilde{r}_n)$ satisfy the (A3) conditions and the SBC property. Furthermore $B(\tilde{x},\tilde{r}_n)\subset B(\tilde{x},r_n)$, and hence  $S_n(x)>\tilde{S}_n(x)$, where 
$\tilde{S}_n(x)=\sum_{k\leq n}1_{\tilde{B}_n}(x)$ diverges to infinity (by the SBC property). 
Thus $\nu\{X_n\geq u_n,\mathrm{i.o}\}=1$ for the corresponding sequence $u_n=\psi(r_n)$, and the result follows.
\end{proof}
If (A3) holds, then we can strengthen the estimates of Theorem \ref{thm.max.log}, especially Case 2.
\begin{cor}
Suppose that $(f,\XX,\nu)$ is a measure preserving system with ergodic
SRB measure $\nu$, and that the local dimension $d_{\nu}$ exists
at $\tilde{x}\in\XX$ and (A2) holds for $\tilde{x}$. Suppose that $\Theta(n)=O(n^{-\zeta})$ and (A1) 
holds together with (A3a) at $\tilde{x}$, with $(2\delta^{-1}+1)/(\zeta-1)<\min\{\alpha,\gamma\}$. 
Consider the observable $\phi(x)=-\log(\dist(x,\tilde{x}))$. Then
\begin{equation}\label{eq.max.exp2}
\lim_{n\to\infty}\frac{M_n(x)}{\log n}=\frac{1}{d_{\nu}},\quad\textrm{a.s.}
\end{equation}
\end{cor}
\begin{proof}
From Case 2 of Proposition \ref{prop.bc.srt}, equation \eqref{eq.srt.poly} applies for a sequence of balls $B_n$ 
satisfying $\nu(B_n)=(\log n)^{\beta}/n$, for any $\beta>0$. For this sequence we can repeat the proof
of Proposition \ref{prop.Mn.bounds} to get the required result.
\end{proof}

For systems having quantitative recurrence estimates we can also establish non-existence of an almost sure
growth rate for certain observables. This extends the results of Theorem \ref{thm.nolimit}. In particular we do
not require the strong assumption of decay of correlations of BV against $L^1$. 
\begin{cor}\label{cor.srt.nolimit}
Suppose that $(f,\XX,\nu)$ satisfies the assumptions of Proposition \ref{prop.bc.srt}, and in addition suppose 
$\nu$ is absolutely continuous with respect to Lebesgue measure. Consider the observable
$\phi(x)=\dist(x,\tilde{x})^{-\alpha}$ for $\alpha>0$. Then for any monotone sequence $u_n\to\infty$ we have:
\begin{equation}\label{eq.mn.nolimit.cor}
\nu\left(\limsup_{n\to\infty}\frac{M_n(x)}{u_n}=0\right)=1,\;\textrm{or}\;\nu\left(
\limsup_{n\to\infty}\frac{M_n(x)}{u_n}=\infty\right)=1.
\end{equation}
\end{cor}
The proof follows immediately from Corollary \ref{cor.srt.io} above, and the method of proof of Theorem \ref{thm.nolimit}.

As an application, Corollary \ref{cor.srt.nolimit} applies to is the family of intermittent maps discussed
in Section \ref{sec.applications}, namely equation \eqref{eq.intermittent}.
%Note that this Corollary does not contradict the counter-example mentioned in 
%\cite{Kim}. For balls $B_n$ that shrink around the fixed point $\tilde{x}=0$, condition (A3) fails to hold.

\subsubsection*{Application: Rank 1 Young towers}
The assumption (A3) on lack of  short return times allows us to deduce a broader class of sequences $r_n$ for which
the corresponding balls $B_n$ form Borel Cantelli sequences. The following result resolves a question posed
in \cite[Section 7]{HNPV}. 

\begin{thm}\label{thm.tower2}
Suppose that $f:\XX\to \XX$ is a $C^2$ diffeomorphism modelled by a Young tower with SRB measure $\nu$ and having exponential tails. Suppose that $\sum_n \nu(B(\tilde{x},r_n))=\infty$, for some monotone sequence $r_n\to 0$,
and $\tilde{x}\in\XX$. 
Consider the observable function $\phi(x)=\psi(\dist(x,\tilde{x}))$ with $\psi(y)$ monotone decreasing,
and $\lim_{y\to 0}\psi(y)=\infty$. Then for $\nu$-a.e. $\tilde{x}\in\XX$:
$$\nu\{X_n\geq \psi(r_n),\mathrm{i.o.}\}=1.$$
\end{thm}
To prove this theorem we check that condition (A3b) holds for $\nu$-a.e. $x\in\XX$.
We then achieve SBC results for the sequence of balls $B_n=B(\tilde{x},r_n)$ with 
$\limsup_n \nu(B_n)(\log n)^{\frac{\sigma}{\zeta}}<\infty$ for some $\sigma>1$.
We then apply Corollary \ref{cor.srt.io} to balls shrinking around such points $\nu$-a.e. $x\in\XX$.
The verification of (A3b) is a technical argument utilizing results in~\cite{CC},
and is postponed to Section~\ref{sec.appendix}. An immediate example to which  this theorem applies is the H\'enon map, as discussed in Section~\ref{sec.applications}.

\section{Discussion.}

We considered two types of function $\phi_1 (x)=-\log d (x,\tilde{x})$ and $\phi_2 (x)=d(x,\tilde{x})^{-\alpha}$. If $X_i$ is a sequence of iid random variables with the same 
distribution function as $\phi_1$ then $M_n/\log n\to 1$ almost surely, while in the case of iid random variables with the distribution function of $\phi_2$ there is no almost sure limit for 
$M_n/u_n$ for any sequence of constants $u_n$. Given sufficient control on dynamical Borel-Cantelli lemmas we were able to show the same phenomena in the deterministic case. In particular, unlike distributional extreme limits the almost sure behavior of $M_n$ does not depend (in many cases) on  the periodicity or not of $\tilde{x}$. Some of our results made assumptions on the recurrence properties
associated to the point $\tilde{x}$ (such as assumption (A3)) but these assumptions were to obtain estimates on dynamical Borel Cantelli lemmas and we believe are an artifact of our method of proof.
In other words we have no reason to believe that if (A3) does not hold at a point $\tilde{x}$  then the almost sure behavior of $M_n$  would differ from the case in which it does. But this is conjecture.

Natural questions include:  can we determine  corresponding almost sure growth results for observables other than those which are functions of distance to a distinguished point? or whose
level set geometry is more complex? This problems requires an understanding of BC and SBC results for more general shrinking targets $B_n$. Notice that our results in Section \ref{sec.results} required the checking of (A1) and (A2). However, for general SRB measures, we would need an extended version of (A2) that is adapted to the geometry of $B_n$. In general, such a condition would be difficult to check and perhaps news ideas are needed.

\section{Proof of Propositions \ref{prop.bc2} and \ref{prop.bc.srt}}\label{sec.prop.proofs}
In this section we prove Propositions \ref{prop.bc2} and \ref{prop.bc.srt}. The main ideas are presented in
\cite{HNPV}, but in this case we need refined estimates on the error terms that appear in the strong Borel Cantelli
results. In the following section we recall some key constructions and estimates that will be used in the proof
of the propositions.
\subsection{Asymptotic estimates for sequences}
To get SBC results together with refined asymptotic estimates the key result we use is the following:
\begin{prop}[\cite{Sprindzuk}] \label{prop:sprindzuk}
  Let $(\Omega,\mathcal{B},\mu)$ be a probability space and let $f_k
  (\omega) $, $(k=1,2,\ldots )$ be a sequence of non-negative $\mu$
  measurable functions and $g_k$, $h_k$ be sequences of real numbers
  such that $0\le g_k \le h_k \le 1$, $(k=1,2, \ldots,)$.  Suppose there exist $C>0$ such that 
  \begin{equation} \label{eq:sprindzuk}
    \int \left(\sum_{m<k\le n}( f_k (\omega) - g_k)
    \right)^2\,d\nu \le C \sum_{m<k \le n} h_k
  \end{equation}
  for arbitrary integers $m <n$. Then for any $\epsilon>0$
  \[
  \sum_{1\le k \le n} f_k (\omega) =\sum_{1\le k\le n} g_k (\omega)  +
  O (\theta^{1/2} (n) \log^{3/2+\epsilon} \theta (n)
  )
  \]
  for $\nu$ a.e.\ $\omega \in \Omega$, where $\theta (n)=\sum_{1\le k
    \le n} h_k$.
\end{prop}
Recall that $S_n (x)=\sum_{j=0}^{n-1} 1_{B_j}(x)$ and $E_n=\sum_{j=0}^{n-1} \nu (B_j)$, we will apply
the proposition to the case where $f_k(x)=1_{B_k}(x)$ and $g_k=\nu(B_k)$, and obtain an estimate for $h_k$
in terms of $\nu(B_k)$, In fact we show that there is a constant $\beta_1<2$ such that
$\theta(n)\leq E^{\beta_1}_n$. It follows that 
for $\nu$-a.e. $x\in\XX$:
  \begin{equation}
  S_n(x) =E_n  +
  O (E^{\beta'}_{n}),
  \end{equation}
for some $\beta'<1$. We try to optimize the constant $\beta'$ in various cases.

\paragraph{Lipschitz approximation of $f_k$.} For the indicator functions $f_k$ defined above,
we would like to apply condition (A1) to estimate the expectations $E(f_if_j)$ (for various $i,j\leq n$). Here
$E(\varphi)=\int \varphi \, d\nu$ for any integrable function $\varphi \in L^1 (\nu)$.

Decay of correlations as stated in condition (A1), requires the observable class to be Lipschitz. 
The function $f_k$, as chosen when applying Proposition \ref{prop:sprindzuk} will be an indicator functions on $B_k$.
Thus we will need to work with a Lipschitz approximation $\tilde{f}_k$ of $f_k$ in order to use (A1). 
The measure of the set of points where $\tilde{f}_k(x)\neq f_k(x)$ will need to be controlled, as this will use
condition (A2). Formally, for each $k$ let $\tilde{f}_k$ be a Lipschitz function such that
$\tilde{f}_k (x)=1_{B_k} (x)$ if $x\in B_k$, $\tilde{f}_k (x)=0$ if
$d(B_k, x)>(k(\log k)^2)^{-1/\delta}$, $0\le \tilde{f}_k \le 1$ and
$\|\tilde{f}_k\|_{\mathrm{Lip}} \le (k(\log k)^{2})^{1/\delta}$.  Clearly we
may construct such functions by linear interpolation of $1_{B_k}$ and
$0$ on a region $r \le d(p_k, x)\le r+ (k(\log k)^2)^{-1/\delta}$.

In Proposition \ref{prop:sprindzuk} we will take $f_k =\tilde{f}_k\circ T^k (x)$, and
$g_k=E(\tilde{f}_k)$. The constants $h_k$ will be
chosen later.

Note that for $\nu$ a.e.\ $x\in \Omega$, $f_i (x) = 1_{B_i} (T^i x)$
except for finitely many $i$ by the Borel--Cantelli lemma as $\nu (x:
f_i (x) \not = 1_{B_i} (T^i x))= \nu (x: r< d (T^i x, p_i) < r+
(i(\log i)^2)^{-1/\gamma}) < (i(\log i)^2)^{-1}$ by assumption (A2).
Furthermore $\sum_k \nu (B_k)=\sum_k g_k +O(1)$.

\paragraph{General estimates for sequences.}
The following result, which is proved in \cite[Lemma 3.1]{HNPV} will be of use in quantifying the asymptotics of
sums of various weighted sequences. For strictly positive sequences, we write $a_n\prec b_n$ if 
$\lim_{n\to\infty}a_n/b_n=0$
\begin{lemma}\label{lem.sequence} (I) Let $\beta\in(0,1)$ and
$a_j\geq cj^{-\beta}$ be a sequence of real numbers. Then for any $\sigma>0$ and
$0<\eta<\frac{1-\beta}{1-\beta+\sigma}$ we have:
\begin{equation}
\left(\sum_{j=1}^{n}j^{\sigma}a_j\right)^{\eta}\prec\sum_{j=1}^{n}a_j.
\end{equation}
(II)
Let $a_j\geq c\frac{(\log j)^{\beta}}{j}$, $(c>0)$, be a sequence of real numbers, where $\beta>0$. Then for any $\sigma>0$,
and $0<\eta<\frac{1+\beta}{1+\beta+\sigma}$ we have for all $n\geq 1$:
\begin{equation}
\left(\sum_{j=1}^{n}a_j(\log j)^{\sigma}\right)^{\eta}\prec\sum_{j=1}^{n}a_j.
\end{equation}
\end{lemma}

\subsection{Proof of Proposition \ref{prop.bc2}}

We first treat the case where condition (A2) holds, $\Theta(n)=O(n^{-\zeta})$ and $\nu(B_k)\geq k^{-\beta}$
for some $\beta<1$. The proof follows \cite{HNPV} and it suffices to give the key steps that lead us
to an estimate on the constant $\beta'$ as specified in Proposition \ref{prop.bc2}. 
A rearrangement of terms in \eqref{eq:sprindzuk} means that it suffices to show that there exists a $C>0$ such that
\[
\sum_{i=m}^n \sum_{j=i+1}^n E(f_i f_j)-E(f_i)E(f_j) \le C \sum_{i=m}^n h_i
\]
for arbitrary integers $n>m$ (where $h_i$ will be chosen later).
We split each sum $\sum_{j=i+1}^n E(f_i f_j)-E(f_i)E(f_j)$ into the terms 
$$
 I_i(n):=\sum_{j=i+1}^{i+\Delta} E(f_i f_j)-E(f_i)E(f_j),\;\textrm{and,}\;
II_i(n):=\sum_{j=i+\Delta+1}^{n} E(f_i f_j)-E(f_i)E(f_j).
$$
We shall put $\Delta=[i^\sigma]$ (where $[\cdot]$ denotes integer part), and specify $0<\sigma<1$ later in the proof. For $II_i(n)$, we can bound
this by a function $h^{II}_i(n)$ with
\begin{equation*}
\begin{split}
h^{II}_i(n) &=\sum_{j=i+[i^\sigma]}^{\infty} E(f_i f_j)-E(f_i)E(f_j)\\
&\leq \sum_{j=i+[i^\sigma]}^{\infty}\|\tilde{f}_i\|_{\mathrm{Lip}}\|\tilde{f}_j\|_{\mathrm{Lip}}\Theta(j-i)\\
&\leq \sum_{k=0}^{\infty} O(1)\left[(i+[i^\sigma]+k)(\log^2(i+[i^\sigma]+k))\right]^{\frac{1}{\delta}}
\left[i\log^2 i\right]^{\frac{1}{\delta}}\left([i^\sigma]+k \right)^{-\zeta}\\
&\leq O(1) i^{\frac{2}{\delta}-\sigma\zeta+\sigma+\epsilon},
\end{split}
\end{equation*}
where $\epsilon>0$ is arbitrary. We set $\theta_2(n)=\sum_{i=1}^{n}h^{II}_i$ and note that 
$\theta_2(n)\prec E^{\kappa}_n$ with
\begin{equation}\label{eq.bd.zeta}
\kappa=\frac{2\delta^{-1}-\sigma\zeta+\sigma+\epsilon+1}{(1-\beta)}.
\end{equation}
To bound $I_i(n)$ set
$$h^{I}_j=\sum_{i=[j-j^{\sigma}]}^{j-1}E(f_i f_j)-E(f_i)E(f_j).$$
A rearrangement of terms shows that
\begin{equation*}
\sum_{i=1}^{n}I_{i}(n) \leq \sum_{j=1}^{n}h^{I}_j
\leq \sum_{j=1}^{n}j^{\sigma}\left(\nu(B_j)+\frac{1}{j\log^2j}\right)
\leq E^{\frac{1}{\eta}}_n+E^{\frac{\sigma}{1-\beta}}_n,
\end{equation*}
where $\eta$ is specified from Lemma \ref{lem.sequence}, 
and in fact can be chosen so that $0<\eta<\frac{1-\beta}{1-\beta+\sigma}$.
If we set $\theta_1(n)=\sum_{i=1}^{n}h^{I}_i$, then we can apply Proposition \ref{prop:sprindzuk} with
$\theta(n)=\theta_1(n)+\theta_2(n)$. To get the best bound on $\beta'$ in Proposition \ref{prop.bc2} we choose 
$\sigma=\zeta^{-1}(2\delta^{-1}+\epsilon+\beta)$ to minimize $\max\{\kappa,1/\eta\}$. For this value, Proposition
\ref{prop:sprindzuk} implies that
\begin{equation*}
  S_n(x) =E_n  +
  O (E^{\beta'}_{n}.
  ),
  \end{equation*}
where $\beta'$ is the constant stated in Proposition \ref{prop.bc2}.
This completes the proof of Case 1.

\medskip

Consider now Case 2 of Proposition \ref{prop.bc2} with $\Theta(n)=O(\theta^{n}_0)$. We argue along a similar lines, and 
for some $\sigma>0$ put
$\Delta=[(\log n)^{\sigma}]$ in the definitions of $I_i(n)$ and $II_i(n)$. As before
we can bound $II_i(n)$ by $h^{II}_i(n)$, where
\begin{equation*}
\begin{split}
h^{II}_i(n) &=\sum_{j=i+[(\log i)^{\sigma}]}^{\infty} E(f_i f_j)-E(f_i)E(f_j)\\
&\leq \sum_{j=i+[(\log i)^{\sigma}]}^{\infty}\|\tilde{f}_i\|_{\mathrm{Lip}}\|\tilde{f}_j\|_{\mathrm{Lip}}\Theta(j-i)\\
&\leq O(1) i^{\frac{2}{\delta}+\epsilon}\theta^{(\log i)^{\sigma}}_0.
\end{split}
\end{equation*}
The latter estimate follows by taking $\Theta(n)=O(\theta^{n}_0)$ in (A1). 
If we set $\theta_2(n)=\sum_{i=1}^{n}h^{II}_i$, then we see that $\theta_2(n)\leq C$ for some uniform constant $C>0$. 
We now turn to $I_i(n)$, and set
$$h^{I}_j=\sum_{i=[j-(\log j)^{\sigma}]}^{j-1}E(f_i f_j)-E(f_i)E(f_j).$$
A rearrangement of terms implies that
\begin{equation*}
\sum_{i=1}^{n}I_{i}(n) \leq \sum_{j=1}^{n}h^{I}_j
\leq \sum_{j=1}^{n}(\log j)^{\sigma}\left(\nu(B_j)+\frac{1}{j\log^2j}\right)
\leq E^{\frac{1}{\eta}}_n+E^{\frac{\sigma-1}{1-\beta}}_n,
\end{equation*}
where $\eta$ satisfies $0<\eta<\frac{1+\beta}{1+\beta+\sigma}$. Given $\beta>0$, we choose $\sigma$ so that
$1<\sigma<1+\beta$, and simultaneously minimize the error term $E^{\frac{1}{\eta}}_n$.  
If we set $\theta_1(n)=\sum_{i=1}^{n}h^{I}_i$, then we can apply Proposition \ref{prop:sprindzuk} with
$\theta(n)=\theta_1(n)+\theta_2(n)$, and obtain:
\begin{equation*}
  S_n(x) =E_n  +
  O (E^{\beta'}_{n}),
  ,
  \end{equation*}
valid for any $\beta'>\frac{2+\beta}{2(1+\beta)}.$
This completes the proof of Proposition \ref{prop.bc2} 

\subsection{Proof of Proposition \ref{prop.bc.srt}}
In the case of having the short return time conditions (A3)
we follow the approach used in proving Proposition \ref{prop.bc2}, where in this case some steps simplify. We take functions $f_k(x)$ and sequences
$g_k$ as before. The only difference is that we apply the condition (A3) to estimate $h_k$. Define
$I_i(n)$ and $II_i(n)$ as before. We consider first Case 1 of Proposition \ref{prop.bc.srt}. In
the definitions of $I_i(n)$ and $II_i(n)$ we set $\Delta=[(\log i)^{\sigma}]$ for some $\sigma>1$.
The estimate $h^{II}_i$ for $II_{i}(n)$ is the same as before, and in fact leads to a constant bound
for $\theta_2(n)=\sum_{i=1}^{n} h^{II}_i.$ For $|j-i|<\Delta$, condition (A3b) leads us to the estimate:
\begin{equation}
E(f_i f_j)-E(f_i)E(f_j)<E(f_i f_j)<C\nu(B_i)^{1+\alpha}.
\end{equation}
For $\gamma$ defined in (A3b), take $\sigma<\gamma$. Then we have:
\begin{equation}
\sum_{i=1}^{n}I_{i}(n)\leq \sum_{j=1}^{n}h^{I}_j\leq \sum_{j=1}^{n}(\log j)^{\sigma}\nu(B_j)^{1+\alpha}\\
\end{equation}
This sum is bounded by $E_n$ if for given $\sigma>1$ we have 
$\nu(B_j)\leq (\log j)^{-\frac{\sigma}{\alpha}}$. It follows that
$\theta_2(n)=\sum_{i=1}^{n} h^{I}_i\leq E_n.$
Hence we can apply Proposition \ref{prop:sprindzuk} with
$\theta(n)=\theta_1(n)+\theta_2(n)$, and obtain for arbitrary $\epsilon>0$:
\begin{equation*}
  S_n(x) =E_n  +
  O (E^{\frac{1}{2}}(\log E_n)^{\frac{3}{2}+\epsilon}
  ).
  \end{equation*}
This completes the proof for Case 1.

\medskip

For Case 2, set $\Delta=[i^{\sigma}]$ for some $\sigma>0$ specified later. We estimate $h^{II}_i(n)$ 
using polynomial decay of correlations and obtain an identical estimate as in the proof of Proposition
\ref{prop.bc.srt}. If we set $\theta_2(n)=\sum_{i=1}^{n} h^{II}_i,$ then $\theta_2(n)$ is uniformly bounded
provided $h^{II}_i$ is summable. This is true if:
$$\sigma>\frac{2\delta^{-1}+1}{\zeta-1}.$$
To estimate $h^{I}_i$, we can apply condition (A3a) for $\sigma<\gamma$, and achieve:
\begin{equation}
\sum_{i=1}^{n}I_{i}(n)\leq \sum_{j=1}^{n}h^{I}_j\leq \sum_{j=1}^{n}j^{\sigma}\nu(B_j)^{1+\alpha}\\
\end{equation}
This sum is bounded by $E_n$ if $\nu(B_j)\leq j^{-\frac{\sigma}{\alpha}}$. It then follows that
$\theta_2(n)=\sum_{i=1}^{n} h^{I}_i\leq E_n.$
Hence we can apply Proposition \ref{prop:sprindzuk} with
$\theta(n)=\theta_1(n)+\theta_2(n)$, and obtain for arbitrary $\epsilon>0$:
\begin{equation*}
  S_n(x) =E_n  +
  O (E^{\frac{1}{2}}(\log E_n)^{\frac{3}{2}+\epsilon}
  ).
  \end{equation*}
This completes the proof for Case 2.

\begin{rmk}
In the proof of Proposition \ref{prop.bc.srt} it is possible to extend the range of $\sigma$ if we relax 
the requirement that $\theta_2(n)$ is bounded. However, we get a weaker asymptotic estimate in equation
\eqref{eq.srt.poly}
\end{rmk}

\section{Appendix: regularity of measures and Condition (A3) for Young towers} \label{sec.appendix}
In this section we prove Theorems \ref{thm.tower} and \ref{thm.tower2}. The proof of Theorem \ref{thm.tower} relies on geometrical properties of the invariant measure for rank one Young towers. The proof of Theorem \ref{thm.tower} requires
the verification of condition (A3b).

\subsubsection*{Proof of Theorem \ref{thm.tower}}
From the Young tower construction \cite{Young}, it follows that condition (A1) applies for $\Theta(n)=O(\theta^{n}_0)$,
and $\theta_0<1$. The following result gives an estimate on how $\nu$ scales on shrinking annuli. It is slightly weaker than condition (A2).
\begin{lemma}\label{lem.annulus.tower}
Suppose that $(f,\XX,\nu)$ satisfies the assumption of Theorem \ref{thm.tower}. There exists
$\hat\delta,\delta>0$ such that for all $x\in\XX$ and all $\epsilon<r^{\hat\delta}$:
\begin{equation}\label{eq.annulus.tower}
\nu(B(x,r+\epsilon))-\nu(B(x,r))\leq O(1)\epsilon^{\delta}.
\end{equation}
The constant $O(1)$ depends on $\tilde{x}$, but can be taken uniform for all $r<r_0(x)$, where $r_0$ depends
on $x$.
\end{lemma} 
\begin{proof}
The following estimate is proved in \cite[Proposition 4.2]{CC}:
For $\nu$-a.e. $x\in\XX$, for all $\delta_1>0, \mathfrak{p}>0$ and $r<r_0(x)$:
\begin{equation}\label{eq.annulus.est1}
\nu(B(x,r))-\nu(B(x,r-r^{\delta_1}))\leq O(1)\left(r^{\delta_1/2}r^{-c_1\mathfrak{p}}+\mathfrak{p}^2r^{-c_2+c_3\mathfrak{p}}
\right),
\end{equation}
where $c_i$ are uniform constants.
Since $\delta_1$ is arbitrary, we set $\epsilon=r^{\delta_1}$.  The bounds can be expressed in terms of a positive 
exponent of $\epsilon$ if $\mathfrak{p}$ satisfies $\frac{\delta_1}{2c_1}>\mathfrak{p}>\frac{c_3}{c_2}.$ This will be
true for all $\delta_1$ sufficiently large.
Formally, we choose $\delta_0>0$ such that 
$\frac{\delta_1}{2c_1}>3\frac{c_3}{c_2},$
for all $\delta_1>\delta_0$. We then choose $\mathfrak{p}=2\frac{c_3}{c_2}$.
For this choice, a corresponding (uniform) $\delta>0$ exists, so that equation \ref{eq.annulus.tower} is valid
for all $\epsilon<r^{\delta_1}$.
\end{proof}
To prove Theorem \ref{thm.tower}, the only modification is in the proof of Proposition \ref{prop.bc2}. We must
instead use a finer sequence of Lipschitz approximation functions $\tilde{f}_k$ to each sequence $f_k$,
and apply Lemma \ref{lem.annulus.tower} to gain an estimate of $h^{II}_i$. For such a sequence, the asymptotics
of the series for $h^{II}_i$ are overall unaffected, in the sense that $\theta_2(n)$ remains summable (by exponential
decay of correlations). This completes the proof.

\subsubsection*{Proof of Theorem \ref{thm.tower2}}
The key estimate in proving Corollary \ref{thm.tower2} is the verification of (A3b). We can use the result of
Lemma \ref{lem.annulus.tower} to replace (A2) as already discussed. To prove Theorem \ref{thm.tower2}
it suffices to prove the following result:
\begin{lemma}
Suppose that $(f,\XX,\nu)$ satisfies the conditions of Theorem \ref{thm.tower2}. Then there exists $\alpha>0$
and $\gamma>1$ such that condition (A3b) for $\nu$-a.e. $x\in\XX$.
\end{lemma}
\begin{proof}
From \cite[Proposition 4.1]{CC}, the following estimate
is achieved: there exist $s_1>0$ and $\alpha_1>0$ such that for all $\gamma_1>1$ the 
following estimate holds on a set $\tilde{\XX}\subset\XX$ with 
$\nu(\XX\setminus\tilde\XX)\leq Cr^{s_1},$ 
\begin{equation}
\sum_{\ell=0}^{|\log r|^{\gamma_1} }\nu(B_{r}(x)\cap T^{-\ell}(B_{r}(x)))\leq
O(1)r^{\alpha_1}\nu(B_r(x)).
\end{equation}
Using this estimate, we can verify condition (A3b). Let 
$\mathcal{M}_r=\XX\setminus\tilde{\XX}$, be such that $\nu(\mathcal{M})\leq Cr^{s_1}$.
Consider first the set $\mathcal{E}_{r,s}$ defined by:
$$\mathcal{E}_{r,s}=\{x:\,\nu(B_{2r}(x))>r^{-s}\nu(B_r(x))\}.$$ 
From \cite[Lemma A.2]{CC}, it can be shown by set covering arguments that for all $r,s>0$
$\nu(\mathcal{E}_{r,s})\leq Cr^{s},$ with $C>0$ uniform.

Consider the sequence $a_n=\frac{1}{n^a}$ for some $a>1$, and let $r\in[a_{n+1},a_n]$. Given $s_1 $ as above we fix the constant $a$
so that $\sum_{n}a^{s}_n<\infty$ for $s<s_1$. Hence going along this sequence, it follows
by the first Borel Cantelli Lemma, that there exists $N_{x}>0$, such that for all $n\geq N$, and $\nu$-a.e. $x\in X$,
we have that $x\not\in\mathcal{M}_{a_n}\cup\mathcal{E}_{a_n,s}$. Hence for $r=a_n$, $n\geq N$, and 
$p<(\log(1/r))^{\gamma_1}$ we have that
\begin{equation}
\nu(B_{r}(x)\cap T^{-\ell}(B_{r}(x)))\leq 
O(1)\nu(B_r(x))r^{\alpha_1}.
\end{equation}
For this value of $x$, we now extend this inequality for all $r\in[a_{n+1},a_n]$ and $n\geq N$. 
Since $\lim_{n\to\infty}a_{n+1}/a_n=1$, it follows that there exists $r_0(x)$ such that for all $r<r_0$ we have
\begin{equation}\label{eq.condC}
\begin{split}
\nu(B_{r}(x)\cap T^{-\ell}(B_{r}(x))) &\leq O(1)\nu(B_{a_n}(x))a^{\alpha_1}_{n}\\
&\leq O(1)(a_{n}/2)^{-s}\nu(B_{a_{n}/2}(x))(2r)^{\alpha_1}\\
&\leq O(1)r^{\alpha_1-s}\nu(B_r(x)).
\end{split}
\end{equation}
By local dimension estimates, we also have for $\nu$-a.e. $x\in X$, the existence of $\epsilon>0$ such that
for all $r<r_{\epsilon}$:
$$r^{d_{\nu}+\epsilon}\leq\nu(B_r(x))\leq r^{d_{\nu}-\epsilon}.$$
Hence if we choose $s$ small enough in equation \eqref{eq.condC} we find that there exists $\alpha>0$ and $\gamma>1$ such
that for all $r<r_0$:
\begin{equation}
\nu(B_{r}(x)\cap T^{-\ell}(B_{r}(x)))\leq 
\nu(B_r(x))^{1+\alpha},
\end{equation}
with $\ell=1,\ldots |\log r|^{\gamma}$. Since $r$ is arbitrary condition (A3b) holds along any sequence $r_n\to 0$.
\end{proof}

\end{document}